\documentclass[12pt]{article}
\usepackage{amsfonts}
\usepackage{graphicx}

\usepackage{amssymb,amsbsy,amsmath,amsfonts,amssymb,amscd,amsthm}
\usepackage{stmaryrd}
\usepackage{mathrsfs}
\usepackage[active]{srcltx}
\usepackage{bbm}
\usepackage{tikz}

\usepackage{paralist}
\usepackage{graphics} 
\usepackage{epsfig} 
\usepackage{epstopdf} 
\usepackage[colorlinks=true]{hyperref}
\usepackage{multirow}
\hypersetup{urlcolor=blue, citecolor=red}

 \newtheorem{theorem}{Theorem}[section]
\newtheorem{lemma}{Lemma}[section]
\newtheorem{remark}{Remark}[section]
\newtheorem{corollary}{Corollary}[section]
\newtheorem{definition}{Definition}[section]
\newtheorem{proposition}{Proposition}[section]

\def\be{\begin{equation}}
\def\ee{\end{equation}}
\def\bea{\begin{eqnarray}}
\def\eea{\end{eqnarray}}
\def\ba{\begin{array}}
\def\ea{\end{array}}

\def\bt{\begin{theorem}}
\def\et{\end{theorem}}
\def\bl{\begin{lemma}}
\def\el{\end{lemma}}
\def\br{\begin{remark}}
\def\er{\end{remark}}
\def\bc{\begin{corollary}}
\def\ec{\end{corollary}}
\def\bd{\begin{definition}}
\def\ed{\end{definition}}
\def\bp{\begin{proposition}}
\def\ep{\end{proposition}}

\def\b{\beta}

\def\k{\kappa}

\def\f{\phi}

\def\o{\omega}

\def\cA{{\cal A}}

\def\cH{{\cal H}}

\def\cO{{\cal O}}

\def\Re {\mathcal {R}e}

\def\ii{{\mbox{i}}}

\def\dis{\displaystyle}

\newcommand{\ol}[1]{\overline{#1}}

\begin{document}

 \title{
Stability of the Timoshenko Beam Equation with One
{\color{black} Weakly Degenerate Local}  Kelvin-Voigt Damping
\thanks{
This work was supported by  National Natural Science Foundation of China (grants No.12271035) and Beijing Natural Science Foundation (grant No. 1232018).
\medskip} }
\author{Ruijuan Liu$^1$\thanks{Corresponding author, email: 3120205705@bit.edu.cn} and Qiong Zhang$^1$\\
		$^{1}$School of Mathematics and Statistics, Beijing Key \\Laboratory
on MCAACI, Beijing Institute of Technology, \\Beijing, 100081,   P. R. China }

\maketitle



\begin{center}
\begin{minipage}{5.5in}

\noindent

 {\bf Abstract.}
We consider the Timoshenko beam equation with locally distributed Kelvin-Voigt damping, which affects either the shear stress or the bending moment. The damping coefficient exhibits a singularity, causing its derivative to be discontinuous. By using the frequency domain method and multiplier technique, we prove that the associated semigroup is polynomial stability. Specifically, regardless of whether the local Kelvin-Voigt damping acts on the shear stress or the bending moment, the system  decays   polynomially with rate  $t^{-\frac{1}{2}}$.

 \vskip 2mm



\end{minipage}
\end{center}

\vskip 10mm


\section{Introduction}
\setcounter{equation}{0} \setcounter{theorem}{0}
\setcounter{lemma}{0} \setcounter{definition}{0}
\setcounter{proposition}{0} \setcounter{remark}{0}
\setcounter{corollary}{0} \setcounter{equation}{0}

In this paper, we consider the following Timoshenko beam equation with local viscoelastic damping of Kelvin-Voigt:
\begin{equation}
\label{system}
\left\{\begin{array}{lcl}
 \rho_1 w_{tt} -  \big[\kappa_1\,\big(w'+ \phi\big)
+ D_{1} \,(w'+ \phi )_{t}\big]' =0 &\mbox{in}& (-1,1)\times{\mathbb R}^+,
\\ \noalign{\medskip}   \displaystyle
 \rho_2 \phi_{tt}
- \big(\kappa_2 \,\phi' +  D_{2} \,\phi_{t}'\big)'
 + \kappa_1\,\big(w'+ \phi\big) +D_{1} \,\big(w'+ \phi \big)_{t} =0 &\mbox{in}& (-1,1)\times{\mathbb R}^+,
\\ \noalign{\medskip}   \displaystyle
 w(x,0) = w_0(x), \; w_t(x,0) = w_1(x) \;&\mbox{  in }&  \;(-1,1),
\\ \noalign{\medskip}   \displaystyle
\phi(x,0) = \phi_0(x), \;
\phi_t(x,0) = \phi_1(x) \quad  \; &\mbox{  in }&  \;(-1,1),
\\ \noalign{\medskip}   \displaystyle
 w(-1,\,t)= w(1,\,t)=\phi(-1,\,t)= \phi(1,\,t)=0\quad  \;  &\mbox{  in }&  \;{\mathbb R}^+.
 \end{array}\right.
\end{equation}
where the  parameters $\rho_1,\;\rho_2,\;\kappa_1,\;\kappa_2$ are positive and one of damping coefficient functions $D_1(\cdot)$ or $D_2(\cdot)$ is zero.
The physical meanings of the functions and  parameters in \eqref{system} are explained in  \cite{Timoshenko}.
Functions $D_1(\cdot),\,D_2(\cdot)$  belong to $L^\infty(-1,1)$ and satisfy
\begin{equation}\label{H1}
D_i(x) =0 \;  \mbox{  for }  \;    x\in[-1,0],
 \quad  D_i(x)=a_i(x) \;  \mbox{for}\;  x\in (0,1],\; i=1,2.
 \tag{H1}\end{equation}
 Let $a_i\in C([0,1])\cap C^1 ((0,1])$ be a function satisfying the following assumption:
\begin{equation}
\label{H2}
\left\{\begin{array}{lcl}
(i)\quad a_i(x)>0,\; \forall x\in(0,1],\; a_i(0)=0
\\ \noalign{\medskip}   \displaystyle
(ii)\quad 0\leq\alpha_i\doteq\sup\limits_{0<x\leq1}\frac{x|a_i^\prime(x)|}{a_i(x)}<1.
\tag{H2}
 \end{array}\right.
\end{equation}

The energy of the system (\ref{system})
is defined by
\begin{equation}
\label{energy}
E(t) = {1\over2}\int_{-1}^1 \Big(\kappa_1|w' + \f|^2
+\kappa_2|\f_x|^2 + \rho_{1}|w_t|^2 +\rho_{2}|\f_t|^2
\Big) dx.
\end{equation}
It is easy to check that
\begin{equation}\label{energydiss}
{{d}\over{dt}}E(t) = -\int_{-1}^{1} \big[D_{1} \,|w_{t}' + \phi_t|^2 +
D_{2} \,|\phi_{t}'  |^2\big] dx.
\end{equation}
Thus, the system (\ref{system}) is dissipative.

In recent decades, a large number of research has focused on studying the stability properties of the Timoshenko beam equation under various damping (see \cite{Almeida, Alves, Tavares, Akil3, Contreras, Cavalcanti, Enyi, Sare, Guesmia, Jorge Silva, Keddi, Khemmoudj, Rivera,  Messaoudi, Mercier, Wehbe1}).    Among these,
 Kelvin-Voigt damping is a viscoelastic damping possessing both elasticity and viscosity. In \cite{Huang}, exponential stability and analyticity of the elastic systems with global Kelvin-Voigt damping were proved. Notably, the stability properties  of the wave equation with local Kelvin-Voigt damping depend on both the position of the damping and the regularity of the damping coefficient (\cite{Akil, Akil1, Akil2, Akil4, Chen, Liu2, Wehbe, Zhang}).

The stability analysis of the Timoshenko beam with local Kelvin-Voigt damping was first addressed in \cite{Zhao}, where exponential stability was obtained when damping coefficient functions are strictly positive on a subinterval of the spatial domain $[0,L]$ and   $C^{1,1}([0;L])$-continuous. Later, Tian and Zhang (\cite{Tian}) introduced more general smoothness conditions on the coefficient functions $D_1(\cdot),\;D_2(\cdot)$ and obtained   exponential stability  or polynomial stability. Liu and Zhang (\cite{Liu}) further studied the relationship between the decay rate of stability and the smoothness properties of the constitutive law. Specifically, they proved that the solution of system \eqref{system} with $D_1(x),\;D_2(x) $ satisfying  \eqref{H1} and  \eqref{H2} is differentiable if $\alpha_1,\; \alpha_2>1$, exponentially stable if $\alpha_1,\; \alpha_2\geq1$ and $D_1^3(x)\leq CD_2(x)$, and polynomially stable of order $t^{-\alpha}$ $(\alpha=\min\limits_{i=1,2}{(1-\alpha_i)^{-1}})$ if $0\leq \alpha_1,\; \alpha_2<1$.
Recently, a sharper polynomial decay rate $t^{-\frac{2-\alpha}{1-\alpha}} \;( \alpha=\min\{\alpha_1,\alpha_2\})$ was obtained for   system \eqref{system} by  applying some Hardy-type inequalities (\cite{R.Liu}).


{\color{black} In \cite{Tavares}, the author studied the Timoshenko system with a global viscoelastic
	dissipation  coupled on the shear force
	\begin{equation}
	\left\{\begin{array}{lcl}
	\displaystyle
	\rho_1 w_{tt} -  \kappa_1\big(w'+ \phi\big)'
	+\kappa_1\int_0^tg(t-s)\big(w'+ \phi \big)'(s)ds =0 &\mbox{in}& (0,L)\times{\mathbb R}^+,
	\\ \noalign{\medskip}   \displaystyle
	\rho_2 \phi_{tt}
	- \kappa_2 \phi''
	+ \kappa_1\,\big(w'+ \phi\big) -\kappa_1 \int_0^tg(t-s)\big(w'+ \phi \big)'(s)ds=0 &\mbox{in}& (0,L)\times{\mathbb R}^+,
	\end{array}\right.
	\end{equation}	
	and obtained that  uniform stability holds if and only if the wave speeds are equal.
	On the other hand, the Timoshenko system with one global viscoelastic
	Kelvin-Voigt damping is polynomially stable with optimal decay rate
	$t^{-\frac{1}{2}}$ (\cite{Malacarne}).
	Moreover, this result also holds for the Timoshenko system with one local
	non-smooth Kelvin-Voigt damping (\cite{Wehbe1}).
An natural problem is: how about the Timoshenko system with one degenerate local Kelvin-Voigt damping.}

In this paper, we analyze the Timoshenko beam system \eqref{system}, where the coefficient functions $D_1(\cdot)$, $D_2(\cdot)$   satisfy assumptions \eqref{H1}-\eqref{H2} and  one of the coefficient functions is assumed to vanish. We obtain that the solution of the system \eqref{system} decays polynomially with the   rate   $t^{-\frac{1}{2}}$, regardless of whether the local viscoelastic damping affects the shear stress or the bending moments. Notably, the same polynomial decay rate has been previously proved for the Timoshenko beam system with one local non-continuous Kelvin-Voigt damping, as established in \cite{Wehbe1}.
Our approach depends on the frequency domain method  (\cite{Borichev}), and the classic multiplier method (\cite{Komornik}). To deal with the singularity of the constitutive law, we introduce cut-off functions as multipliers and use proper inequalities.

This paper is organized as follows: In Section 2, we present our main results and provide some preliminaries. Section 3 is dedicated to proving the main results. {\color{black}The conclusion is given in Section 4.} Throughout the paper, we denote $C$ a generic positive constant.

\section{Main Results and Preliminaries}
\setcounter{equation}{0} \setcounter{theorem}{0}
\setcounter{lemma}{0} \setcounter{definition}{0}
\setcounter{proposition}{0} \setcounter{remark}{0}
\setcounter{corollary}{0} \setcounter{equation}{0}

In this section, we shall present the main results and preliminaries.
First, introducing the energy space
$$
\cH = H_0^1(-1,1) \times H_0^1(-1,1) \times L_2(-1,1) \times L_2(-1,1)
$$
with norm
$$
\|(w,\f, v,\psi)\|_{\cH}^2 =
 \int_{-1}^1 \big(\kappa_1|w' + \f|^2
+\kappa_2|\f'|^2 + \rho_{1}|v|^2 +\rho_{2}|\psi|^2\;\big)dx
$$
for all $(w,\f, v,\psi)\in \cH.$

Define a linear unbounded operator $\cA \:: D(\cA) \subset\cH \to \cH$ by
\begin{equation}
\label{def-a}
\begin{array}{l}\dis
\cA (w,\f, v,\psi) = \Big(
v,\;\psi, {1\over\rho_{1}}T', {1\over{\rho_{2}}} (R'-T)
\Big),  \quad
\\ \noalign{\medskip}   \displaystyle
D(\cA) = \Big\{(w,\f, v,\psi)\in \cH\: \Big|\:
v,\psi \in H_0^1(-1,1), \:
T',\, R' \in L^2(-1,1)
\Big\},
\end{array}
\end{equation}
where $T = \kappa_1(w'+ \f)+ D_{1}(v'+ \psi),  \;
R  = \kappa_2 \phi'+D_{2} \psi'$.

Then system \eqref{system} can be written as the evolution equation
\begin{equation}\label{evol}
U_t = {\cal A} \, U, \quad \forall t>0,\;
U(0) = U_0=(w_0, \, w_1, \, \phi_0, \,\phi_1) \in {\cal H}.
\end{equation}
By the same arguments as {\color{black}Theorem 2.1 in \cite{Tian} or Corollary 2.1 in \cite{Zhao}}, we can get the following well-posedness result for the system \eqref{system}.
\begin{lemma}\label{th-well}   Assume coefficient functions $D_{1}$ and $D_{2}$ are continuous, nonnegative and satisfy
\eqref{H1}-\eqref{H2}. Then, the  operator  $\cA$  generates a C$_0$-semigroup of contractions $e^{t{\cal A}}$ on $\cH$, and $\ii\mathbb{R}\subset \rho(\cA)$, the resolvent set of $\cA.$
\end{lemma}
In order to prove the main result, we need the following lemma.
\begin{lemma}\label{bt} (\cite{Borichev,Rao})
Let $A\,:\,D(A) \subset H \to H$ generate a bounded $C_0$-semigroup $e^{tA}$ on   Hilbert space  $H$.
Assume that
\begin{equation}
    \label{lem-spec}
 \ii\omega \in   \rho(A), \quad   \forall \; \omega\in \mathbb{R}.
   \end{equation}
Then the semigroup $e^{tA}$ is  polynomially stable of order $\displaystyle  {1\over\beta} $ if and only if
\begin{equation}\label{lem-poly}
\varlimsup\limits_{\omega\in \mathbb{R},|\omega|\to \infty}|\omega|^{-\beta} \big\|(\ii\omega I-A)^{-1}\big\|_{{\cal L}(H)} < \infty.
\end{equation}
\end{lemma}

The main result of this paper can be stated as follows.
\begin{theorem}\label{th-main}
Suppose the coefficient functions $D_1(\cdot),\,D_2(\cdot)$  satisfy  \eqref{H1}-\eqref{H2} and   that one of them vanishes. Then, the   semigroup $e^{t\cA}$ associated with \eqref{evol} is polynomially stable  of order $\frac{1}{2}$, i.e.,
 \begin{equation}
 \|e^{\cA t}U_0\|_\cH\leq\frac{C}{\sqrt t}\|U_0\|_{D(\cA)}.
 \end{equation}
 \end{theorem}

\section{Polynomial Stability Analysis}
\setcounter{equation}{0} \setcounter{theorem}{0}
\setcounter{lemma}{0} \setcounter{definition}{0}
\setcounter{proposition}{0} \setcounter{remark}{0}
\setcounter{corollary}{0} \setcounter{equation}{0}

In this section, we will prove the  Theorem \ref{th-main}. From Lemma \ref{bt}, this is equivalent to showing that
there exists constant $r>0$ such that
\begin{equation}\label{huang2}
 \inf\limits_{\|U\|_{\cal H}=1,\: \o\in{\mathbb R} } \o^\beta\|\ii\o U-{\cal A}U\|_{\cal H}  \ge r,
\end{equation}
with the parameter $\beta=2$.

Suppose that \eqref{huang2} fails. Then
 there exists   $\{ \o_n,  U_n  \}_{n=1}^\infty
\subset {\mathbb R} \times D({\cal A}) $ with
\begin{equation}\label{unitnorm}
\begin{array}{l}
\|U_n\|_{\cal H}  =\|  (w_n,\phi_n, v_n,\psi_n)\|_{\cal H}  =1
\;  \mbox{ and  }  \;
\o_n\to \infty\end{array},
\end{equation}
such that
 \begin{equation}
  \label{ps00}
  \o_n^{\beta}\|\ii\o_n U_n- \cA U_n\|_{\cH} =o(1),
\end{equation}
which is equivalent to
\begin{subequations}
\begin{align}
\label{ps1}
   &f_{1n}  \doteq  \omega_n^\beta (\ii  \omega_n w_{n} -  v_{n}) =o(1)& &\mbox{in}\;\; H^1_{0}(-1,\,1),
   \\ \noalign{\medskip} \label{ps2}
   & f_{2n}  \doteq  \omega_n^\beta( \ii  \omega_n \phi_{n} -  \psi_{n} )=o(1)& &\mbox{in}\;\; H^1_{0}(-1,\,1),
 \\ \noalign{\medskip} \label{ps3}
 &f_{3n}  \doteq  \omega_n^\beta ( \ii  \omega_n v_{n} -  \rho_1^{-1} T_n' )=o(1) & &\mbox{in}\;\; L^2(-1,\,1),
 \\ \noalign{\medskip} \label{ps4}
 &f_{4n} \doteq  \omega_n^\beta  ( \ii  \omega_n \psi_n - \rho_2^{-1}(R_n' -T_n)) =o(1) & &\mbox{in}\;\;  L^2(-1,\,1),
\end{align}
\end{subequations}
where
$$
T_n =   \kappa_1(w_n' +  \phi_n) +D_{1}( v_n' + \psi_n) ,  \qquad
R_n =  \kappa_2\, \phi_n' +D_{2} \psi_n' .
$$
Substituting \eqref{ps1} into \eqref{ps3} and \eqref{ps2} into \eqref{ps4}, respectively,
we deduce that
\begin{align}
\label{ma12}
 &     -\omega_n^2 w_{n} -  \rho_1^{-1} T_n' =\ii\o_n^{1-\b}f_{1n}+\o_n^{-\b}f_{3n} & &\mbox{in}\;\; L^2(-1,\,1),
 \\ \noalign{\medskip} \label{ma13}
 &  -\omega_n^2 \phi_n - \rho_2^{-1}(R_n' -T_n) =\ii\o_n^{1-\b}f_{2n}+\o_n^{-\b}f_{4n} & &\mbox{in}\;\;  L^2(-1,\,1).
\end{align}
Furthermore, by the dissipativeness of  operator $\cA$  and  \eqref{ps00},
\begin{equation}
\label{ps50}
 \o_n^\b \Re\langle (\ii\o_n I - \cA)U_n, U_n\rangle_{\cal H}
=\o_n^\b \big(
\|D_{1}^{1\over2}(v'_n+ \psi_n)\|_{L^2(-1,1)}^2+ \|D_{2}^{1\over2}\psi_n'\|_{L^2(-1,1)}^2\big)
=o(1).
\end{equation}
Then, from \eqref{H1} and \eqref{ps50}, we conclude that
\begin{equation}
 \label{ps5}
 \|a_1^{1\over2} (v'_n+\psi_n) \|_{L^2(0,1)}, \quad
 \|a_2^{1\over2}\psi_n'\|_{L^2(0,1)} =\o_n^{-{\b\over2}} o(1).
 \end{equation}
It is clear from \eqref{ps1}, \eqref{ps2} and \eqref{ps5} that
\begin{equation}
 \label{ps6}
 \|a_1^{1\over2} (w'_n+\phi_n) \|_{L^2(0,1)}, \quad
 \|a_2^{1\over2}\phi_n'\|_{L^2(0,1)} =\o_n^{-1-{\b\over2}} o(1).
 \end{equation}
Furthermore, from \eqref{unitnorm}, it is easy to see that
\begin{equation}\label{bounded}
\|w_n'\|_{L^2(-1,1)},\;\|v_n\|_{L^2(-1,1)},\;\|\f_n'\|_{L^2(-1,1)},\;\|\psi_n\|_{L^2(-1,1)}=\cO(1),
\end{equation}
which together with \eqref{ps1}, \eqref{ps2} and the fact that $\omega_n\to \infty,$ yields
\begin{equation}\label{o1}
\|w_n\|_{L^2(-1,1)},\quad \|\phi_n\|_{L^2(-1,1)}
=o(1).
\end{equation}

 For any $\displaystyle0<\varepsilon<\frac{1}{10}$, we introduce the following cut-off functions
{\color{black}
\begin{enumerate}
\item[{\rm (i)}] $\displaystyle p_i(x)\in  C^1\left[\frac{\varepsilon}{4},1-\frac{\varepsilon}{4}\right],\, i=1,2,3$ such that $0\leq p_i(x)\leq1$ and
  \begin{equation}
  p_i(x)=
   \left\{\begin{array}{lcl}
  1 &\mbox{if}&\displaystyle  x\in \Big[\frac{(2i+1)\varepsilon}{4},
  1-\frac{(2i+1)\varepsilon}{4}\Big],
\\ \noalign{\medskip}   \displaystyle
  0 &\mbox{if}&\displaystyle x\in \Big[\frac{\varepsilon}{4},\frac{(2i-1)\varepsilon}{4}
  \Big]\cup \Big[1-\frac{(2i-1)\varepsilon}{4},1-\frac{\varepsilon}{4}\Big].
 \end{array}\right.
   \end{equation}

\item[{\rm (ii)} ] $q_j(x)\in C^1[-1,1]$ such that $0\leq q_j(x)\leq1$ ($j=1,2$) and
   \begin{equation}
   \begin{array}{lcl}\displaystyle
 q_1(x)=
   \left\{\begin{array}{lcl}
  1 &\mbox{if}&\displaystyle  x\in\Big[-1,\frac{7\varepsilon}{4}\Big],
\\ \noalign{\medskip}   \displaystyle
  0 &\mbox{if}&\displaystyle  x\in\Big[1-\frac{7\varepsilon}{4},1\Big].
 \end{array}\right.
 \;\;\;
  q_2(x)=
   \left\{\begin{array}{lcl}
  0 &\mbox{if}&\displaystyle  x\in\Big[-1,\frac{7\varepsilon}{4}\Big],
\\ \noalign{\medskip}   \displaystyle
  1 &\mbox{if}& \displaystyle x\in\Big[1-\frac{7\varepsilon}{4},1\Big].
 \end{array}\right.
 \end{array}
 \end{equation}
 \end{enumerate}}

\subsection{The case $a_1(x)>0,\;a_2(x)=0$}
In this subsection, we   analyze the stability of the system \eqref{system} with $a_1(x)>0$ and $a_2(x)=0$.

\begin{lemma}\label{le2}
If $a_2(x)=0$, the following  holds
\begin{equation}\label{ma1}
\int_{\frac{5\varepsilon}{4}}
^{1-\frac{5\varepsilon}{4}}|w_n^{\prime}|^2dx=o(1),\quad
\int_{\frac{3\varepsilon}{4}}^{1-\frac{3\varepsilon}{4}}|\o_nw_n|^2dx=o(1).
\end{equation}
\end{lemma}
\noindent{\bf Proof.} First,   from \eqref{ps5} and $\|\psi_n\|_{L^2(0,1)}$ is bounded, one has
\begin{equation}\label{ma41}
\|a_1^{\frac{1}{2}}v_n^{\prime}\|_{L^2(0,1)}=\cO(1).
\end{equation}
Multiplying \eqref{ps3} with $-\ii\omega_n^{-1-\b}\rho_1p_1\ol{v_n}$ and integrating over {\color{black}$(\frac{\varepsilon}{4},1-\frac{\varepsilon}{4})$}, we arrive at
\begin{equation}\label{ma2}
\rho_1\int_{\frac{\varepsilon}{4}}^{1-\frac{\varepsilon}{4}}p_1|v_n|^2dx+
\ii\omega_n^{-1}\int_{\frac{\varepsilon}{4}}^{1-\frac{\varepsilon}{4}}\big[\kappa_1(w_n^{\prime}+\phi_n)+a_1(v_n^{\prime}+\psi_n)\big]^{\prime}p_1\ol{v_n}dx
=\o_n^{-1-\b}o(1).
\end{equation}
By  \eqref{ps5}-\eqref{o1}, \eqref{ma41},
  we get
\begin{equation}\label{ma3}
\begin{array}{lcl}
 &&\displaystyle
\ii\omega_n^{-1}\int_{\frac{\varepsilon}{4}}^{1-\frac{\varepsilon}{4}}\big[\kappa_1(w_n^{\prime}+\phi_n)+a_1(v_n^{\prime}+\psi_n)\big]^{\prime}
p_1\ol{v_n}dx
 \\ \noalign{\medskip} &=&\displaystyle
-\ii\omega_n^{-1}\int_{\frac{\varepsilon}{4}}^{1-\frac{\varepsilon}{4}}\big[\kappa_1(w_n^{\prime}+\phi_n)+a_1(v_n^{\prime}+\psi_n)\big]
(p_1^{\prime}\ol{v_n}+p_1\ol{v_n^{\prime}})dx
\\ \noalign{\medskip} &\leq&\displaystyle
C|\omega_n^{-1}|\big[\|w_n^{\prime}+\phi_n\|_{L^2(0,1)}\|v_n\|_{L^2(0,1)}
+\|a_1^{-\frac{1}{2}}(w_n^{\prime}+\phi_n)\|_{L^2(\frac{\varepsilon}{4},1-\frac{\varepsilon}{4})}\|a_1^{\frac{1}{2}}v_n^{\prime}\|_{L^2(0,1)}
\\ \noalign{\medskip} &&+\displaystyle
\|a_1^{\frac{1}{2}}(v_n^{\prime}+\psi_n)\|_{L^2(0,1)}\big(\|a_1^{\frac{1}{2}}v_n\|_{L^2(0,1)}
+\|a_1^{\frac{1}{2}}v_n^{\prime}\|_{L^2(0,1)}\big)\big]=o(1).
\end{array}
\end{equation}
Thus, substituting \eqref{ma3} into \eqref{ma2}, and using the denfition of $p_1(x)$, yields
\begin{equation}\label{0}
\int_{\frac{3\varepsilon}{4}}^{1-\frac{3\varepsilon}{4}}|v_n|^2dx=o(1).
\end{equation}
The second estimation in \eqref{ma1} is clear from \eqref{ps1} and \eqref{0}.

Now, multiplying \eqref{ps4} by $\o_n^{-\b}\rho_2p_2\ol{(w_n^{\prime}+\phi_n)}$ and integrating over {\color{black}$\big(\frac{3\varepsilon}{4},1-\frac{3\varepsilon}{4}\big)$}, we obtain
\begin{equation}
\int_{\frac{3\varepsilon}{4}}^{1-\frac{3\varepsilon}{4}}\big[\ii\omega_n\rho_2\psi_n-(\kappa_2\phi_n^{\prime})^{\prime}
+\kappa_1(w_n^{\prime}+\phi_n)+a_1(v_n^{\prime}+\psi_n)\big]p_2\ol{(w_n^{\prime}+\phi_n)}dx=\o_n^{-1-\b}o(1).
\end{equation}
Using \eqref{ps5}, $\|w_n^{\prime}+\phi_n\|_{L^2(-1,1)}$ is bounded (see \eqref{bounded}) and the definition of $p_2$, we get
\begin{equation}\label{1}
\int_{\frac{3\varepsilon}{4}}^{1-\frac{3\varepsilon}{4}}\big[\ii\omega_n\rho_2\psi_np_2\ol{(w_n^{\prime}+\phi_n)}
-\kappa_2\phi_n^{\prime\prime}p_2\ol{(w_n^{\prime}+\phi_n)}+
\kappa_1p_2|w_n'+\phi_n|^2\big]dx
=\o_n^{-1-\b}o(1).
\end{equation}
Moreover, since $\|\psi_n\|_{L^2(-1,1)},\;\o_n^{-1}\|\phi_n^{\prime\prime}\|_{L^2(-1,1)}$
is  bounded due to \eqref{ps4} and \eqref{bounded}, using \eqref{H2}, \eqref{ps6} yields
\begin{equation}
\ii\omega_n\rho_2\int_{\frac{3\varepsilon}{4}}^{1-\frac{3\varepsilon}{4}}\psi_np_2\ol{(w_n^{\prime}+\phi_n)}dx
\le C|\omega_n|\|a_1^{\frac{1}{2}}(w_n^{\prime}+\phi_n)\|_{L^2(0,1)}\|a_1^{-\frac{1}{2}}\psi_n\|_{L^2(\frac{3\varepsilon}{4},1-\frac{3\varepsilon}{4})}
=\o_n^{-\frac{\b}{2}}o(1),
\end{equation}
and
\begin{equation}
\kappa_2\int_{\frac{3\varepsilon}{4}}^{1-\frac{3\varepsilon}{4}}\phi_n^{\prime\prime}p_2\ol{(w_n^{\prime}+\phi_n)}dx
\le
C\|a_1^{\frac{1}{2}}(w_n^{\prime}+\phi_n)\|_{L^2(0,1)}
\|a_1^{-\frac{1}{2}}\phi_n^{\prime\prime}\|_{L^2(\frac{3\varepsilon}{4},1-\frac{3\varepsilon}{4})}
=\o_n^{-\frac{\b}{2}}o(1),
\end{equation}
Combining these  with \eqref{1} and the definition of $p_2$, we have
\begin{equation}
\int_{\frac{5\varepsilon}{4}}^{1-\frac{5\varepsilon}{4}}|w_n^{\prime}+\phi_n|^2dx=o(1).
\end{equation}
From the above estimation and \eqref{o1}, we obtain
 \begin{equation}
 \int_{\frac{5\varepsilon}{4}}^{1-\frac{5\varepsilon}{4}}|w_n^{\prime}|^2dx
 \le
 2\int_{\frac{5\varepsilon}{4}}^{1-\frac{5\varepsilon}{4}}|w_n^{\prime}+\phi_n|^2dx
 +2\int_{\frac{5\varepsilon}{4}}^{1-\frac{5\varepsilon}{4}}|\phi_n|^2dx=o(1),
 \end{equation}
which completes the proof of Lemma \ref{le2}.\hfill$\Box$

\begin{lemma} \label{le3}For $a_2(x)=0$, one has the following estimates
\begin{equation}\label{ma4}
\int_{\frac{5\varepsilon}{4}}^{1-\frac{5\varepsilon}{4}}|\omega_n\phi_n|^2dx,
\quad\int_{\frac{7\varepsilon}{4}}^{1-\frac{7\varepsilon}{4}}|\phi_n^{\prime}|^2dx=o(1).
\end{equation}
\end{lemma}
\noindent{\bf Proof.} Multiplying \eqref{ps3} by $\o_n^{-\beta}\rho_1p_2\ol{\phi_n^{\prime}}$ and integrating over {\color{black}$\big(\frac{3\varepsilon}{4},1-\frac{3\varepsilon}{4}\big)$} yields:
\begin{equation}\label{ma5}
\int_{\frac{3\varepsilon}{4}}^{1-\frac{3\varepsilon}{4}}\big[\ii\omega_n\rho_1p_2v_n\ol{\phi_n^{\prime}}-[\kappa_1(w_n^{\prime}+\phi_n)+a_1(v_n^{\prime}+\psi_n)]^{\prime}p_2\ol{\phi_n^{\prime}}\;\big]dx=\o_n^{-\b}o(1).
\end{equation}
Using integration by parts, along with \eqref{ps2} and \eqref{0}, and the fact that $\|\psi_n\|_{L^2(-1,1)}$ is bounded, we have
\begin{equation}\label{ma6}
\begin{array}{lcl}\displaystyle
\ii\omega_n\rho_1\int_{\frac{3\varepsilon}{4}}^{1-\frac{3\varepsilon}{4}}p_2v_n\ol{\phi_n^{\prime}}dx
&=&\displaystyle
-\ii\omega_n\rho_1\int_{\frac{3\varepsilon}
	{4}}^{1-\frac{3\varepsilon}{4}}p_2v_n^{\prime}\ol{\phi_n}dx
+\rho_1\int_{\frac{3\varepsilon}{4}}^{1-\frac{3\varepsilon}{4}}p_2^{\prime}v_n\ol{(\psi_n+\o_n^{-\b}f_{2n})}dx
\\ \noalign{\medskip} &=& \displaystyle
 -\ii\omega_n\rho_1\int_{\frac{3\varepsilon}{4}}^{1-\frac{3\varepsilon}{4}}p_2v_n^{\prime}\ol{\phi_n}dx+o(1).
\end{array}
\end{equation}
Using \eqref{H2}, \eqref{ps2}, \eqref{ps5}, and \eqref{o1}, we can easily see that:
\begin{equation}\label{ma7}
\begin{array}{lcl}
&&\displaystyle -\ii\omega_n\rho_1\int_{\frac{3\varepsilon}{4}}^{1-\frac{3\varepsilon}{4}}p_2v_n^{\prime}\ol{\phi_n}dx
\\ \noalign{\medskip}  &=& \displaystyle
-\ii\omega_n\rho_1\int_{\frac{3\varepsilon}{4}}^{1-\frac{3\varepsilon}{4}}p_2a_1^{\frac{1}{2}}(v_n^{\prime}+\psi_n)a_1^{-\frac{1}{2}}\ol{\phi_n}dx
-\rho_1\int_{\frac{3\varepsilon}{4}}
^{1-\frac{3\varepsilon}{4}}p_2|\omega_n\phi_n|^2dx+\o_n^{\frac{\beta}{2}}o(1)
\\ \noalign{\medskip} &=& \displaystyle
-\rho_1\int_{\frac{3\varepsilon}{4}}
^{1-\frac{3\varepsilon}{4}}p_2|\omega_n\phi_n|^2dx+\o_n^{1-\frac{\beta}{2}}o(1).
\end{array}
\end{equation}
Then, substituting \eqref{ma7} into \eqref{ma6}, we obtain
\begin{equation}\label{ma8}
 \ii\omega_n\rho_1\int_{\frac{3\varepsilon}{4}}^{1-\frac{3\varepsilon}{4}}p_2v_n\ol{\phi_n^{\prime}}dx
=-\rho_1\int_{\frac{3\varepsilon}{4}}^{1-\frac{3\varepsilon}{4}}p_2|\omega_n\phi_n|^2dx+\o_n^{1-\frac{\beta}{2}}o(1).
\end{equation}
From \eqref{H2}, \eqref{ps5}-\eqref{ps6},
and the fact that $\|\phi_n^{\prime}\|_{L^2(-1,1)}$ and $\o_n^{-1}\|\phi_n^{\prime\prime}\|_{L^2(-1,1)}$ are
bounded, we have
\begin{equation}\label{ma9}
\begin{array}{lcl}&&\displaystyle -\int_{\frac{3\varepsilon}{4}}^{1-\frac{3\varepsilon}{4}}\big[\kappa_1(w_n^{\prime}+\phi_n)+a_1(v_n^{\prime}+\psi_n)\big]^{\prime}p_2\ol{\phi_n^{\prime}}dx
\\ \noalign{\medskip} &=& \displaystyle
\int_{\frac{3\varepsilon}{4}}^{1-\frac{3\varepsilon}{4}}\big[\kappa_1(w_n^{\prime}+\phi_n)+a_1(v_n^{\prime}+\psi_n)\big]
(p_2\ol{\phi_n^{\prime\prime}}+p_2^{\prime}\ol{\phi_n^{\prime}})dx
\\ \noalign{\medskip} &\leq& \displaystyle
C\big[\|a_1^{\frac{1}{2}}(w_n^{\prime}+\phi_n)\|_{L^2(0,1)}\big(\|a_1^{-\frac{1}{2}}\phi_n^{\prime\prime}\|_{L^2(\frac{3\varepsilon}{4},1-\frac{\varepsilon}{2})}
+\|a_1^{-\frac{1}{2}}\phi_n^{\prime}\|_{L^2(\frac{3\varepsilon}{4},1-\frac{3\varepsilon}{4})}\big)
\\ \noalign{\medskip} &&+\displaystyle
\|a_1^{\frac{1}{2}}(v_n^{\prime}+\psi_n)\|_{L^2(0,1)}\big(\|a_1^{\frac{1}{2}}\phi_n^{\prime\prime}\|_{L^2(0,1)}
+\|a_1^{\frac{1}{2}}\phi_n^{\prime}\|_{L^2(0,1)}\big)\big]=\o_n^{1-\frac{\beta}{2}}o(1).
\end{array}
\end{equation}
By substituting \eqref{ma8}-\eqref{ma9} into \eqref{ma5} and using $\beta=2$, we obtain the first inequality in \eqref{ma4}.

Next, multiplying \eqref{ps4} by $\o_n^{-\beta}\rho_2p_3\ol{\phi_n}$ and integrating over ${\color{black}(\frac{5\varepsilon}{4},1-\frac{5\varepsilon}{4})}$, we get
\begin{equation}\label{e31}
\int_{\frac{5\varepsilon}{4}}^{1-\frac{5\varepsilon}{4}}\big[\ii\omega_n\rho_2\psi_n-\kappa_2\phi_n^{\prime\prime}
+\kappa_1(w_n^{\prime}+\phi_n)+a_1(v_n^{\prime}+\psi_n)\big]p_3\ol{\phi_n}dx=\o_n^{-\b}o(1).
\end{equation}
Using  \eqref{ps2}, \eqref{ps5}, \eqref{o1}, \eqref{e31} and the fact that
$\|w_n^{\prime}+\phi_n\|_{L^2(-1,1)}$ is bounded, one has
\begin{equation}
-\rho_2\int_{\frac{5\varepsilon}{4}}^{1-\frac{5\varepsilon}{4}}p_3|\o_n\phi_n|^2dx
+\kappa_2
\int_{\frac{5\varepsilon}{4}}^{1-\frac{5\varepsilon}{4}}p_3|\phi_n^{\prime}|^2dx
+\kappa_2\int_{\frac{5\varepsilon}{4}}
^{1-\frac{5\varepsilon}{4}}p_3^{\prime}\phi_n\ol{\phi_n^{\prime}}dx= o(1).
\end{equation}
Therefore, combining these with the first estimate in \eqref{ma4},   we obtain the second estimate in \eqref{ma4}. This completes the proof.\hfill$\Box$

\begin{lemma}\label{le4} For $a_2(x)=0$, one has the following result holds
\begin{equation}\label{ma19}
\begin{array}{lcl}&&\displaystyle
\int_{-1}^{\frac{7\varepsilon}{4}} \big(\rho_2|\omega_n\phi_n|^2+\kappa_2|\phi_n^{\prime}|^2+\rho_1|\omega_n w_n|^2+\kappa_1|w_n^{\prime}|^2\big)dx
\\ \noalign{\medskip} &&+\displaystyle
\int_{1-\frac{7\varepsilon}{4}}^{1}\big( \rho_2|\omega_n\phi_n|^2+\kappa_2|\phi_n^{\prime}|^2+\rho_1|\omega_n w_n|^2+\kappa_1|w_n^{\prime}|^2\big)dx=o(1).
\end{array}
\end{equation}
\end{lemma}
\noindent{\bf Proof.}
Multiplying \eqref{ma13} by $2\rho_2h\ol{\phi_n^{\prime}}$ and integrating over $(-1,1)$, where $h\in C^1(-1,1)$ and $h(-1)=h(1)=0$, then using
\eqref{ps5}, \eqref{o1}, and   \eqref{bounded},
we obtain
\begin{equation}\label{ma17}
\begin{array}{lcl}&&\displaystyle
Re \int_{-1}^1\big[h^{\prime}\big(\rho_2|\omega_n\phi_n|^2+\kappa_2|\phi_n^{\prime}|^2\big)+2\kappa_1(w_n^{\prime}+\phi_n)h\ol{\phi_n^{\prime}}\;\big]dx
+2 Re \int_0^1a_1(v_n^{\prime}+\psi_n)h\ol{\phi_n^{\prime}}dx
\\ \noalign{\medskip} &=& \displaystyle
Re \int_{-1}^1\big[h^{\prime}\big(\rho_2|\omega_n\phi_n|^2
+\kappa_2|\phi_n^{\prime}|^2\big)
+2\kappa_1hw_n^{\prime}\ol{\phi_n^{\prime}}\;\big]dx+ o(1)= o(1),
\end{array}
\end{equation}
 Set $S_n\doteq\kappa_1w_n^{\prime}+D_1(v_n^{\prime}+\psi_n),$ {\color{black} from \eqref{ps5} and \eqref{bounded}, the definition of $D_1$, we get that $S_n$
is uniformly
bounded in $L^2(0, 1)$.}
Then,
multiplying both sides of \eqref{ma12} by $2\rho_1\kappa_1^{-1}h\ol{S_n}$ and integrating over $(-1,1)$, we obtain the following estimation by using \eqref{ps5} and the fact that $\|\phi_n^{\prime}\|_{L^2(-1,1)}$ is bounded:
\begin{equation}\label{ma15}
\begin{array}{lcl}&&\displaystyle
Re \int_{-1}^1 \big[-\rho_1\omega_n^2 w_{n}-[\kappa_1(w_n^{\prime}+\phi_n)+D_1(v_n^{\prime}+\psi_n)]^{\prime}\;\big]2\kappa_1^{-1}h\ol{S_n} dx
\\ \noalign{\medskip} &=& \displaystyle
Re\int_{-1}^1 \big(\kappa_1^{-1}h^{\prime}|S_n|^2-2\rho_1\kappa_1^{-1}\omega_n^2 w_{n}h\ol{S_n}-2\kappa_1h\phi_n^{\prime}\ol{w_n^{\prime}}\;\big)dx-
2Re\int_0^1h\phi_n^{\prime}\ol{a_1(v_n^{\prime}+\psi_n)}dx
\\ \noalign{\medskip} &=& \displaystyle
Re\int_{-1}^1 \big(\kappa_1^{-1}h^{\prime}|S_n|^2-2\rho_1\kappa_1^{-1}\omega_n^2 w_{n}h\ol{S_n}\big)dx-2\kappa_1Re\int_{-1}^1h\phi_n^{\prime}\ol{w_n^{\prime}}dx+\o_n^{-\frac{\b}{2}}o(1)
=\o_n^{1-\b}o(1).\end{array}
\end{equation}
By \eqref{ps1}, \eqref{ps5} and $\|v_n\|_{L^2(-1,1)}$ is bounded, we get
\begin{equation}\label{ma16}
\begin{array}{lcl}&&\displaystyle
-2\rho_1\kappa_1^{-1}\omega_n^2\int_{-1}^1  w_{n}h\ol{S_n}dx=-2\rho_1\kappa_1^{-1}\omega_n^2\int_{-1}^1 w_{n}h\ol{\big[\k_1w_n^{\prime}+D_1(v_n^{\prime}+\psi_n)\big]}dx
\\ \noalign{\medskip} &=&\displaystyle
\rho_1\int_{-1}^1 h^{\prime}|\omega_n w_n|^2dx+2\ii\omega_n\rho_1k_1^{-1}\int_0^1h v_{n}\ol{a_1(v_n^{\prime}+\psi_n)}dx+\o_n^{1-\frac{3\beta}{2}}o(1)
\\ \noalign{\medskip} &=& \displaystyle
\rho_1\int_{-1}^1 h^{\prime}|\omega_n w_n|^2dx+\o_n^{1-\frac{\beta}{2}}o(1).
\end{array}
\end{equation}
Finally, inserting \eqref{ma16} into \eqref{ma15} and using \eqref{ma12} along with the fact that $\|w_n^{\prime}\|_{L^2(-1,1)}$ is bounded and $\beta=2$, we obtain
\begin{equation}\label{ma18}
\int_{-1}^1 h^{\prime}\big(\rho_1|\omega_n w_n|^2+\k_1^{-1}|S_n|^2\big)dx-2\kappa_1\int_{-1}^1h\phi_n^{\prime}\ol{w_n^{\prime}}dx=o(1).
\end{equation}
Consequently, we conclude the following estimate by   \eqref{ma17} and \eqref{ma18}:
\begin{equation}\label{ma11}
\int_{-1}^1 h^{\prime}\big(\;\rho_1|\omega_n w_n|^2+\kappa_1^{-1}|S_n|^2+\rho_2|\omega_n\phi_n|^2+\kappa_2|\phi_n^{\prime}|^2\big)dx=o(1).
\end{equation}

\noindent
Now, let $h=(x+1)q_1+(x-1)q_2$ in \eqref{ma11}. Then using the definition of $q_i$ and $D_1$, and combining \eqref{ps5}, \eqref{bounded}, and Lemma \ref{le2}-Lemma \ref{le3}, we deduce that
\begin{equation}\label{ma20}
\begin{array}{lcl}&&\displaystyle
\int_{-1}^{\frac{7\varepsilon}{4}}\big( \rho_2|\omega_n\phi_n|^2+\kappa_2|\phi_n^{\prime}|^2+\rho_1|\omega_n w_n|^2+\kappa_1|w_n^{\prime}|^2\big)dx
\\ \noalign{\medskip} &&\displaystyle
+\int_{1-\frac{7\varepsilon}{4}}^{1} \big(\rho_2|\omega_n\phi_n|^2+\kappa_2|\phi_n^{\prime}|^2+\rho_1|\omega_n w_n|^2+
\kappa_1|w_n^{\prime}|^2\big)dx
\\ \noalign{\medskip} &=&\displaystyle
-\int_{\frac{7\varepsilon}{4}}^{1-\frac{7\varepsilon}{4}} \big[q_1+(x+1)q_1^{\prime}\big]\big(\rho_2|\omega_n\phi_n|^2+\kappa_2|\phi_n^{\prime}|^2+\rho_1|\omega_n w_n|^2+\kappa_1^{-1}|S_n|^2\;\big)dx
\\ \noalign{\medskip} &&- \displaystyle
\int_{\frac{7\varepsilon}{4}}^{1-\frac{7\varepsilon}{4}} \big[q_2+(x-1)q_2^{\prime}\big]\big(\rho_2|\omega_n\phi_n|^2+\kappa_2|\phi_n^{\prime}|^2+\rho_1|\omega_n w_n|^2+\kappa_1^{-1}|S_n|^2\;\big)dx
\\ \noalign{\medskip} &&- \displaystyle
\int_{\left[0,\frac{7\varepsilon}{4}\right]\cup\left[1-\frac{7\varepsilon}{4},1\right]}\big[2w_n^{\prime}a_1(v_n^{\prime}+\psi_n)
+\kappa_1^{-1}|a_1(v_n^{\prime}+\psi_n)|^2\big]dx+o(1)
\\ \noalign{\medskip} &\leq&\displaystyle
C\Big(\|\omega_n\phi_n\|_{L^2(\frac{5\varepsilon}{4},1-\frac{5\varepsilon}{4})}^2
+\|\phi_n'\|_{L^2(\frac{7\varepsilon}{4},1-\frac{7\varepsilon}{4})}^2
+\|\omega_nw_n\|_{L^2(\frac{3\varepsilon}{4},1-\frac{3\varepsilon}{4})}^2
+\|w_n'\|_{L^2(\frac{5\varepsilon}{4},1-\frac{5\varepsilon}{4})}^2
\\ \noalign{\medskip} &&\displaystyle
+\|w_n'\|_{L^2(-1,1)}\|a_1^{\frac{1}{2}}(v_n^{\prime}+\psi_n)\|_{L^2(0,1)}
+\|a_1^{\frac{1}{2}}(v_n^{\prime}+\psi_n)\|_{L^2(0,1)}^2\Big)+o(1)=o(1).
\end{array}\nonumber
\end{equation}
Therefore, \eqref{ma19} is proved. \hfill$\Box$


\subsection{The case $a_1(x)=0,\;a_2(x)>0$}
This subsection is devoted to proving the stability of the system \eqref{system} when $a_1(x)=0$ and $a_2(x)>0$, i.e., when the local viscoelastic damping affects   the bending moment. To achieve this goal, we present the following lemmas.
\begin{lemma} \label{le6}If $a_1(x)=0$, it holds
\begin{equation}\label{ma21}
\|\omega_n\phi_n\|_{L^2(0,1)}
=\omega_n^{-\frac{\beta}{2}}o(1),
\quad\|\phi_n^{\prime}\|_{L^2(0,1)}
=\omega_n^{-\frac{\beta}{2}}o(1).
\end{equation}
\end{lemma}
\noindent{\bf Proof.}
In fact, when $0\le a_2 <1$ and $0\leq\tau<1$, a direct computation gives that
	\begin{equation}
    \begin{array}{lcl}\label{add1}
	|\psi_{n}(\tau)|  &=&\displaystyle \Big|\int_\tau^1 \psi_{n}'dx \Big| \le  \Big(\int_0^1 a_2 | \psi_{n}'|^2 dx \Big)^{1\over2}
	\Big(\int_\tau^1 \frac{1}{a_2} dx \Big)^{1\over2}
    \\ \noalign{\medskip} &\leq& \displaystyle
     \Big(\int_0^1 a_2 | \psi_{n}'|^2 dx \Big)^{1\over2}
	\Big(\int_\tau^1 \frac{1}{a_2(1)x^{\alpha_2}} dx \Big)^{1\over2}.
    \end{array}
	\end{equation}
	Combining this with \eqref{H2} and \eqref{ps5}, and letting $\tau=0$, we obtain
 	\begin{equation}\label{a1}
    |\psi_{n}(0)|=\omega_n^{-\frac{\beta}{2}}o(1).
    \end{equation}
	Moreover, by \eqref{add1}, we have
	\begin{equation}\label{add2}
    \begin{array}{lcl}  \displaystyle
	\int_0^1 |\psi_{n}| ^2dx&\leq &\displaystyle
    \int_0^1 \Big(\int_0^1 a_2 | \psi_{n}'|^2 dx \Big)
	\Big(\int_\tau^1 \frac{1}{a_2} dx \Big)d\tau
    \\ \noalign{\medskip} &=& \displaystyle
    \int_0^1 a_2 | \psi_{n}'|^2 dx\int_0^1\frac{1}{a_2} d\tau\int_0^\tau dx
    \\ \noalign{\medskip} &=& \displaystyle
    \int_0^1 a_2 | \psi_{n}'|^2 dx\int_0^1\frac{\tau}{a_2} d\tau.
    \end{array}
	\end{equation}
From \eqref{H2}, we deduce that
\begin{equation}
 \int_0^1\frac{\tau}{a_2} d\tau
 \leq\frac{1}{a_2(1)}\int_0^1\tau^{1-\alpha_2} d\tau
 =\frac{1}{a_2(1)(2-\alpha_2)},
\end{equation}
which along with \eqref{ps5} and \eqref{add2} yields
    \begin{equation}\label{a2}
    \|\psi_{n}\|_{L^2(0,1)} = \omega_n^{-\frac{\beta}{2}}o(1).
    \end{equation}
Then, by \eqref{ps2} and \eqref{a2}, we have the first estimation of \eqref{ma21}.

\indent
Moreover, multiplying \eqref{ps4} by $\o_n^{-\beta}\rho_2\ol{\phi_n}$, using \eqref{ps2} and
integrating on $(0,1)$, one has
\begin{equation}\label{ma37}
\begin{array}{lcl}
&&\displaystyle
-\rho_2\|\psi_n\|_{L^2(0,1)}^2-R_n\ol{\phi_n}\big|_0^1+\int_0^1\big[(\kappa_2\phi_n'+a_{2}\psi_n')\ol{\phi_n'}+\kappa_1(w_n'+  \phi_n)\ol{\phi_n}\;\big]dx
\\ \noalign{\medskip} &=&\displaystyle
-\rho_2\|\psi_n\|_{L^2(0,1)}^2+\kappa_2\|\phi_n'\|_{L^2(0,1)}^2-R_n\ol{\phi_n}\big|_0^1
+\int_0^1\big[a_{2}\psi_n'\ol{\phi_n'}+\kappa_1(w_n' +  \phi_n)\ol{\phi_n}\;\big]dx
\\ \noalign{\medskip} &=&\displaystyle o(1).
\end{array}
\end{equation}
Now, by using \eqref{ps5}, \eqref{ps6}, the first estimation of \eqref{ma21} and the fact that $\|w_n' +  \phi_n\|_{L^2(-1,1)}$ is bounded, we arrive at
\begin{equation}\label{ma38}
\begin{array}{lcl}&&\displaystyle
\int_0^1\big[a_{2}\psi_n'\ol{\phi_n'}+\kappa_1(w_n' +  \phi_n)\ol{\phi_n}\;\big]dx
\\ \noalign{\medskip} &\leq&\displaystyle
\|a_2^{\frac{1}{2}}\psi_n'\|_{L^2(0,1)} \|a_2^{\frac{1}{2}}\phi_n'\|_{L^2(0,1)}+\kappa_1\|w_n' +  \phi_n\|_{L^2(0,1)}\|\phi_n\|_{L^2(0,1)}
=\omega_n^{-1-\frac{\beta}{2}}o(1).
\end{array}
\end{equation}
It follows from $\phi_n(0)=\omega_n^{-1-\frac{\beta}{2}}o(1)$ due to \eqref{ps2} and \eqref{a1}.
By \eqref{ps4}, \eqref{bounded}, one has $|\phi_n'(0)|\le C(\o_n \|\psi_n\|_{L^2(-1,1)} +\|T_n\|_{L^2(-1,1)} ) + \o_n^{-\beta} o(1) \le \o_n \cO(1). $
Combining these with  $\phi_n(1)=0$, we obtain
\begin{equation}\label{ma39}
R_n\ol{\phi_n}\big|_0^1=
\big[(\kappa_2\phi_n'+a_{2}\psi_n')\ol{\phi_n}\,\big]_0^1
=-\kappa_2\phi_n'(0) \ol{\phi_n (0)}
=\omega_n^{-\frac{\beta}{2}}o(1).
\end{equation}
Therefore, by combining \eqref{ma37}-\eqref{ma39} and \eqref{a2}, we obtain the second estimate in \eqref{ma21}. Hence, we have completed the proof.\hfill$\Box$

\begin{lemma}\label{le7} For $a_1(x)=0$, one has the following estimates
\begin{equation}\label{ma22}
\int_{\frac{5\varepsilon}{4}}^{1-\frac{5\varepsilon}{4}}|w_n^{\prime}|^2dx=o(1),
\quad \int_{\frac{7\varepsilon}{4}}^{1-\frac{7\varepsilon}{4}}|\o_nw_n|^2dx=o(1).
\end{equation}
\end{lemma}
\noindent{\bf Proof.}
Multiplying \eqref{ps4} by $\o_n^{-\b}\rho_2p_2\ol{w_n^{\prime}}$ and integrating over {\color{black}$\big(\frac{3\varepsilon}{4},1-\frac{3\varepsilon}{4}\big)$}, we obtain
\begin{equation}\label{ma23}
\int_{\frac{3\varepsilon}{4}}^{1-\frac{3\varepsilon}{4}}\big[\ii\omega_n\rho_2\psi_n-(\kappa_2\phi_n^{\prime}+a_2\psi_n^{\prime})^{\prime}+\kappa_1(w_n^{\prime}+\phi_n)\big]p_2\ol{w_n^{\prime}}dx=\o_n^{-\b}o(1).
\end{equation}
Using \eqref{o1}, \eqref{ma21}, \eqref{a2} and $\|w_n^{\prime}\|_{L^2(-1,1)}$ is bounded (see \eqref{bounded}), we get
\begin{equation}\label{ma24}
\int_{\frac{3\varepsilon}{4}}^{1-\frac{3\varepsilon}{4}}\big[\kappa_1p_2|w_n^{\prime}|^2-(\kappa_2\phi_n^{\prime}+a_2\psi_n^{\prime})^{\prime}p_2\ol{w_n^{\prime}}\;\big]dx=\o_n^{1-\frac{\beta}{2}}o(1).
\end{equation}
Note that $\|w_n^{\prime}\|_{L^2(-1,1)},\;\o_n^{-1}\|w_n^{\prime\prime}\|_{L^2(-1,1)}$ is bounded
due to \eqref{ps3} and \eqref{bounded}.
Then, by \eqref{H2}, \eqref{ps5}, \eqref{ps6}, \eqref{ma21}, we have
\begin{equation}\label{ma25}
\begin{array}{lcl}&&\displaystyle
-\int_{\frac{3\varepsilon}{4}}^{1-\frac{3\varepsilon}{4}}(\kappa_2\phi_n^{\prime}
+a_2\psi_n^{\prime})^{\prime}p_2\ol{w_n^{\prime}}dx=
\int_{\frac{3\varepsilon}{4}}^{1-\frac{3\varepsilon}{4}}(\kappa_2\phi_n^{\prime}+a_2\psi_n^{\prime})(p_2^{\prime}\ol{w_n^{\prime}}+p_2\ol{w_n^{\prime\prime}}\;)dx
\\ \noalign{\medskip} &\leq&\displaystyle
C\big(\|\phi_n^{\prime}\|_{L^2(0,1)}\|w_n^{\prime}\|_{L^2(-1,1)}
+\| \phi_n^{\prime}\|_{L^2(0,1)}\| w_n^{\prime\prime}\|_{L^2(\frac{3\varepsilon}{4},1-\frac{3\varepsilon}{4})}
\\ \noalign{\medskip} &&\displaystyle
+\|a_2^{\frac{1}{2}}\psi^{\prime}\|_{L^2(0,1)}\|a_2^{\frac{1}{2}}w_n^{\prime}\|_{L^2(0,1)}
+\|a_2^{\frac{1}{2}}\psi_n^{\prime}\|_{L^2(0,1)}\| w_n^{\prime\prime}\|_{L^2(0,1)}\big)
=\o_n^{1-\frac{\b}{2}}o(1).
\end{array}
\end{equation}
Substituting \eqref{ma25} into \eqref{ma24}, we obtain
 \begin{equation}
 \int_{\frac{3\varepsilon}{4}}^{1-\frac{3\varepsilon}{4}}p_2|w_n^{\prime}|^2dx=\o_n^{1-\frac{\b}{2}}o(1).
 \end{equation}
Therefore, we obtain the first estimate in \eqref{ma22} by the above inequality and the assumption $\beta=2$.


Now, multiplying \eqref{ma12} with $\rho_1p_3\ol{w_n}$ and integrating over ${\color{black}(\frac{5\varepsilon}{4},1-\frac{5\varepsilon}{4})}$, one has
\begin{equation}\label{ma26}
\int_{\frac{5\varepsilon}{4}}^{1-\frac{5\varepsilon}{4}}\big[\rho_1p_3|\o_nw_n|^2+[\kappa_1(w_n^{\prime}+\phi_n)]^{\prime}p_3\ol{w_n}\big]dx=\o_n^{1-\b}o(1).
\end{equation}
Combining \eqref{o1},  $\|w_n^{\prime}+\phi_n\|_{L^2(-1,1)}$ is bounded, and the first estimation in \eqref{ma22}, we can get
\begin{equation}\nonumber
\begin{array}{lcl}&&\displaystyle
\int_{\frac{5\varepsilon}{4}}^{1-\frac{5\varepsilon}{4}}[\kappa_1(w_n^{\prime}+\phi_n)]^{\prime}p_3\ol{w_n}dx
=
 -\kappa_1\int_{\frac{5\varepsilon}{4}}^{1-\frac{5\varepsilon}{4}}(w_n^{\prime}+\phi_n)(p_3'\ol{w_n}+p_3\ol{w_n^{\prime}})dx
\\ \noalign{\medskip} &\leq&\displaystyle
C\big(\|w_n^{\prime}+\phi_n\|_{L^2(-1,1)}\|w_n\|_{L^2(-1,1)}
+\|w_n^{\prime}\|_{L^2(\frac{5\varepsilon}{4},1-\frac{5\varepsilon}{4})}^2
+\|\phi_n\|_{L^2(0,1)}\|w_n^{\prime}\|_{L^2(\frac{5\varepsilon}{4},1-\frac{5\varepsilon}{4})}\big)=o(1).
\end{array}
\end{equation}
Substituting these into \eqref{ma26} and using the assumption that $\beta=2$  yields the second estimation in \eqref{ma22}.  \hfill$\Box$

\begin{lemma} \label{le8}For $a_1(x)=0$, we get the following estimate
\begin{equation}\label{ma31}
\int_{\big[-1,\frac{7\varepsilon}{4}\big]\cup\left[1-\frac{7\varepsilon}{4},1\right]} \left(\kappa_1|w_n^{\prime}|^2+\rho_1|\omega_nw_n|^2\right)dx
+\int_{-1}^{0}\big(\kappa_2|\phi_n^{\prime}|^2+\rho_2|\omega_n \phi_n|^2\big)dx=o(1).
\end{equation}
\end{lemma}
\noindent{\bf Proof.}
Multiplying \eqref{ma12} by $2\rho_1h\ol{w_n^{\prime}}$ and integrating over $(-1,1)$, where $h\in C^1(-1,1)$ and $h(-1)=h(1)=0$,
we obtain
\begin{equation}\label{ma33}
\int_{-1}^1 h^{\prime}\big(\rho_1|\omega_nw_n|^2+\kappa_1|w_n^{\prime}|^2\big)dx-2\kappa_1Re \int_{-1}^1h\phi_n^{\prime}
\ol{w_n^{\prime}}dx=\o_n^{1-\b}o(1).
\end{equation}
 From \eqref{bounded} and \eqref{o1}, the definition of $D_2$, we have $R_n$
is uniformly
bounded in $L^2(0, 1)$. Then, multiplying \eqref{ma13} by $2\rho_2\kappa_2^{-1}h\ol{R_n}$ and integrating over $(-1,1)$, and utilizing \eqref{ps5}, \eqref{o1}, \eqref{bounded}, we obtain
\begin{equation}\label{ma28}
\begin{array}{lcl}&&\displaystyle
Re\int_{-1}^1 \big[-\omega_n^2\rho_2 \phi_{n}-(\kappa_2\phi_n^{\prime}+D_2\psi_n^{\prime})^{\prime}+\kappa_1(w_n^{\prime}+\phi_n)\;\big]2\kappa_2^{-1}h\ol{R_n}dx
\\ \noalign{\medskip} &=& \displaystyle
Re\int_{-1}^1 \big(\kappa_2^{-1}h^{\prime}|R_n|^2-2\rho_2\kappa_2^{-1}\omega_n^2\phi_{n}h\ol{R_n}+
2\kappa_1\kappa_2^{-1}(w_n^{\prime}+\phi_n)h\ol{R_n}\; \big)dx
\\ \noalign{\medskip} &=& \displaystyle
Re\int_{-1}^1\big(\kappa_2^{-1}h^{\prime}|R_n|^2-2\rho_2\kappa_2^{-1}\omega_n^2\phi_{n}h\ol{ R_n}\big)dx+2\kappa_1Re\int_{-1}^1hw_n^{\prime}\ol{\phi_n^{\prime}}dx+o(1)
= \o_n^{1-\beta} o(1).
\end{array}
\end{equation}
From \eqref{ps2}, \eqref{ps5}, and \eqref{a2}, it is easy to get
\begin{equation}\label{ma29}
\begin{array}{lcl}&&\displaystyle
 -2\rho_2\kappa_2^{-1}\omega_n^2 Re\int_{-1}^1\phi_{n}h\ol{R_n}dx=-2\rho_2\kappa_2^{-1}\omega_n^2Re\int_{-1}^1  \phi_{n}h\ol{(\kappa_2\phi_n^{\prime}+D_2\psi_n^{\prime})}dx
\\ \noalign{\medskip} &=&\displaystyle
\rho_2\int_{-1}^1 h^{\prime}|\omega_n \phi_n|^2dx+2\omega_n\rho_2\kappa_2^{-1}Re\int_0^1h \ii\psi_{n}a_2\ol{\psi_n^{\prime}}dx+\o_n^{1-\frac{3\b}{2}}o(1)
\\ \noalign{\medskip} &=& \displaystyle
\rho_2\int_{-1}^1 h^{\prime}|\omega_n \phi_n|^2dx+\o_n^{1-\b}o(1).
\end{array}
\end{equation}
Substituting \eqref{ma29} into \eqref{ma28},   utilizing \eqref{ma13}  and the assumption $\b=2$,  we have
\begin{equation}\label{ma30}
\int_{-1}^1 h^{\prime}\big(\rho_2|\omega_n \phi_n|^2+\kappa_2^{-1}|R_n|^2\big)dx+2\kappa_1 Re \int_{-1}^1hw_n^{\prime}\ol{\phi_n^{\prime}}dx=o(1).
\end{equation}
Therefore, by \eqref{ma30} and \eqref{ma33}, we arrive at
\begin{equation}\label{ma34}
\int_{-1}^1 h^{\prime}\big(\kappa_1|w_n^{\prime}|^2+\rho_1|\omega_n w_n|^2+\rho_2|\omega_n\phi_n|^2+\kappa_2^{-1}|\kappa_2\phi_n^{\prime}+D_2\psi_n^{\prime}|^2\big)dx=o(1).
\end{equation}

 Now, taking $h=(x+1)q_1+(x-1)q_2$ in \eqref{ma34}, we deduce that
\begin{equation}\label{ma32}
\begin{array}{lcl}&&\displaystyle
 \int_{-1}^{\frac{7\varepsilon}{4}}\big( \kappa_1|w_n^{\prime}|^2+\rho_1|\omega_n w_n|^2\big)dx+\int_{-1}^{0}\big(\kappa_2|\phi_n^{\prime}|^2+\rho_2|\omega_n\phi_n|^2\big)dx
\\ \noalign{\medskip}&&\displaystyle
+\int_{1-\frac{7\varepsilon}{4}}^{1} \big(\kappa_1|w_n^{\prime}|^2+\rho_1|\omega_n w_n|^2\big)dx
\\ \noalign{\medskip} &=&\displaystyle
-\int_{\frac{7\varepsilon}{4}}^{1-\frac{7\varepsilon}{4}} \big[q_1+(x+1)q_1^{\prime}\big]\big(\kappa_1|w_n^{\prime}|^2+\rho_1|\omega_n w_n|^2+\rho_2|\omega_n\phi_n|^2+\kappa_2^{-1}|\kappa_2\phi_n^{\prime}+a_2\psi_n^{\prime}|^2\big)dx
\\ \noalign{\medskip} &&- \displaystyle
\int_{\frac{7\varepsilon}{4}}^{1-\frac{7\varepsilon}{4}} \big[q_2+(x-1)q_2^{\prime}\big]\big(k_1|w_n^{\prime}|^2+\rho_1|\omega_n w_n|^2+\rho_2|\omega_n\phi_n|^2+\kappa_2^{-1}|\kappa_2\phi_n^{\prime}+a_2\psi_n^{\prime}|^2\big)dx
\\ \noalign{\medskip} &&- \displaystyle
\int_{\left[0,\frac{7\varepsilon}{4}\right]\cup\left[1-\frac{7\varepsilon}{4},1\right]}\big(\phi_n^{\prime}a_2\psi_n^{\prime}+\kappa_2^{-1}|a_2\psi_n^{\prime}|^2
+\kappa_2|\phi_n^{\prime}|^2
+\rho_2|\omega_n\phi_n|^2\big)dx +o(1).
\end{array}\nonumber
\end{equation}
Consequently, we conclude \eqref{ma31} by using Lemma \ref{le6} and Lemma \ref{le7}. The proof of Lemma \ref{le8} is finished. \hfill$\Box$

\noindent{\bf Proof of Theorem \ref{th-main}}
By   Lemma \ref{le2}-\ref{le4} for the case where $a_1(\cdot)>0,\;a_2(\cdot)=0$, and Lemma \ref{le6}-\ref{le8}  for the case where $a_1(\cdot)=0,\;a_2(\cdot)>0$, we deduce that $\|U_n\|_{\cal H}=o(1)$ . This  contradicts \eqref{unitnorm}. Thus, the proof of Theorem \ref{th-main} is complete.\hfill$\Box$

{\color{black}
\section{Conclusion}
\setcounter{equation}{0} \setcounter{theorem}{0}
\setcounter{lemma}{0} \setcounter{definition}{0}
\setcounter{proposition}{0} \setcounter{remark}{0}
\setcounter{corollary}{0} \setcounter{equation}{0}
In this paper, we  have shown that the system \eqref{system} with
 a single local weakly degenerate Kelvin-Voigt damping is   polynomially stable with rate   $t^{-\frac{1}{2}}$.
 The following table provides a summary of the   stability  properties of the semigroup $e^{t{\cal A}}$ corresponding to system \eqref{system} (\cite{Liu,R.Liu,Tian, Wehbe1,Zhao} and references therein):
\begin{center}
	\begin{tabular}{|c|c|c|}
		\hline
	$D_i$  	&$\alpha_i$    &  decay rate \\ \hline
		\multirow{5}{*}{$D_i \neq 0$ }
		&$\alpha_1,\;\alpha_2=0 $  &  polynomially stable of order $t^{-2}$ \\
		\cline{2-3}
		&$\alpha_1,\;\alpha_2\in [0,1)$  & polynomially stable of order $t^{-\frac{2-\alpha}{1-\alpha}}$ with $\alpha\doteq\min\{\alpha_1,\;\alpha_2\}$  \\
		\cline{2-3} &$\alpha_1\in [0,1),\; \alpha_2=1$  &  polynomially stable of order strictly less than $t^{-{1\over 1-\alpha_1}}$\\
		\cline{2-3}&$\alpha_1=1,\; \alpha_2\in [0,1) $   &  polynomially stable of order strictly less than $t^{-{1\over 1-\alpha_2}}$\\ 								
		\cline{2-3}&$\alpha_1,\; \alpha_2 \ge 1$   &exponentially stable \\ \hline
		 $D_1= 0 $ &$\alpha_2\in [0,1)$  &\multirow{2}{*}{polynomially stable of order $t^{-1/2}$} \\ \cline{1-2}
		 $D_2=   0$ &$\alpha_1\in [0,1)$  &
		 \\\hline
	\end{tabular}
\end{center}

There are still some
open questions to be resolved  about
the Timoshenko beam equation with local Kelvin-Voigt damping.
 In particular, it is still unknown whether the polynomial stability order obtained in this paper is optimal, as well as the stability for
 $\alpha_i\geq1$.}

\end{document}